\documentclass[default]{sn-jnl}

\jyear{2021}%

\usepackage{etoolbox}   
\makeatletter
\patchcmd{\@maketitle}{\artauthors}{\centerline{\artauthors}}{}{}
\makeatother

\makeatletter
\patchcmd{\ps@headings}
{\hbox to \hsize{\hfill Springer Nature 2021 \LaTeX\ template\hfill}}
{\hbox to \hsize{}}
{}
{}
\patchcmd{\ps@headings}
{\hbox to \hsize{\hfill Springer Nature 2021 \LaTeX\ template\hfill}}
{\hbox to \hsize{}}
{}
{}
\patchcmd{\ps@titlepage}
{\hbox to \hsize{\hfill Springer Nature 2021 \LaTeX\ template\hfill}}
{\hbox to \hsize{}}
{}
{}
\makeatother

\theoremstyle{thmstyleone}%
\newtheorem{theorem}{Theorem}

\newtheorem{lemma}{Lemma}
\newtheorem{claim}{Claim}

\theoremstyle{thmstyletwo}%

\theoremstyle{thmstylethree}%

\newcommand{\hr}{{\hat{r}}}
\newcommand{\flo}[1]{\lfloor #1 \rfloor}
\newcommand{\cei}[1]{\lceil #1 \rceil}

\begin{document}

\title[Size Ramsey numbers of small graphs versus fans or paths]{Size Ramsey numbers of small graphs versus fans or paths}


\author[1]{\fnm{Yufan} \sur{Li}}\email{yfli@stu.hebtu.edu.cn}

\author*[1,2]{\fnm{Yanbo} \sur{Zhang}}\email{ybzhang@hebtu.edu.cn}

\author[3]{\fnm{Yunqing} \sur{Zhang}}\email{yunqingzh@nju.edu.cn}

\affil[1]{\orgdiv{School of Mathematical Sciences}, \orgname{Hebei Normal University}, \orgaddress{\city{Shijiazhuang}, \postcode{050024}, \state{Hebei}, \country{China}}}

\affil[2]{\orgname{Hebei International Joint Research Center for Mathematics and Interdisciplinary Science}, \orgaddress{\city{Shijiazhuang}, \postcode{050024}, \state{Hebei}, \country{China}}}

\affil[3]{\orgdiv{Department of Mathematics}, \orgname{Nanjing University}, \orgaddress{\city{Nanjing}, \postcode{210093}, \state{Jiangsu}, \country{China}}}

\abstract{For two graphs $G_1$ and $G_2$, the size Ramsey number $\hat{r}(G_1,G_2)$ is the smallest positive integer $m$ for which there exists a graph $G$ of size $m$ such that for any red-blue edge-coloring of the graph $G$, $G$ contains either a red subgraph isomorphic to $G_1$, or a blue subgraph isomorphic to $G_2$. Let $P_n$ be a path with $n$ vertices, $nK_2$ a matching with $n$ edges, and $F_n$ a graph with $n$ triangles sharing exactly one vertex. If $G_1$ is a small fixed graph and $G_2$ denotes any graph from a graph class, one can sometimes completely determine $\hat{r}(G_1,G_2)$. Faudree and Sheehan confirmed all size Ramsey numbers of $P_3$ versus complete graphs in 1983. The next year Erd\H{o}s and Faudree confirmed that of $2K_2$ versus complete graphs and complete bipartite graphs. We obtain three more Ramsey results of this type. For $n\ge 3$, we prove that $\hat{r}(P_3,F_n)=4n+4$ if $n$ is odd, and $\hat{r}(P_3,F_n)=4n+5$ if $n$ is even. This result refutes a conjecture proposed by Baskoro~et al. We also show that $\hat{r}(2K_2,F_2)=12$ and $\hat{r}(2K_2,F_n)=5n+3$ for $n\ge 3$. In addition, we prove that $\hat{r}(2K_2,nP_m)=\min\{nm+1, (n+1)(m-1)\}$. This result verifies a conjecture posed by Vito and Silaban.}




\maketitle

\section{Introduction}\label{sec1}

Let $G$ be a graph (herein, we only consider finite simple graphs without isolated vertices). As usual, we denote by $V(G)$ and $E(G)$, respectively, the set of vertices and the set of edges of $G$. We refer to $\vert V(G)\vert$ as the \emph{order} and to $\vert E(G)\vert$ as the \emph{size} of $G$. For two graphs $G_1$ and $G_2$, we write $G\to (G_1,G_2)$ if for any partition $(E_1,E_2)$ of $E(G)$, either $G_1$ is a subgraph of the graph induced by $E_1$ or $G_2$ is a subgraph of the graph induced by $E_2$. The \emph{size Ramsey number} $\hat{r}(G_1,G_2)$ was introduced by Erd\H{o}s~et al.~\cite{Erdoes1978size} in 1978, where $\hat{r}(G_1,G_2)$ is the smallest positive integer $m$ for which there exists a graph $G$ of size $m$ satisfying $G\to (G_1,G_2)$. In the language of coloring, $\hat{r}(G_1,G_2)$ is the smallest size of $G$ with the property that $G$ always contains either a red copy of $G_1$ or a blue copy of $G_2$ for any red-blue edge-coloring of $G$. Size Ramsey numbers are one of the most well-established topics in graph Ramsey theory. A survey can be found in~\cite{Faudree2002survey}.

Size Ramsey numbers are more difficult to determine compared with (order) Ramsey numbers. A main problem in this area is to determine when $\hat{r}(G,G)$ grows linearly with $\vert V(G)\vert$. This topic has been well studied if $G$ is a path~\cite{Beck1983size,Dudek2017some}, a tree with bounded maximum degree~\cite{Dellamonica2012size,Haxell1995size}, a cycle~\cite{Haxell1995Induced,Javadi2019Size}, etc. On the other hand, there are also a number of papers concerning the exact value of size Ramsey numbers~\cite{Burr1978Ramsey,Davoodi2021Conjecture,Erdoes1984Size,Erdoes1978size,Faudree1983Size,Faudree1983Sizea,Harary1983Generalized,Javadi2018Question,Lortz1998Size,Lortz2021Size,Miralaei2019Size,Sheehan1984Class}.

In the present paper, we focus on determining $\hr(G_1,G_2)$ when $G_1$ has two edges and $G_2$ is any graph in a class of graphs. Since $G_1$ has no isolated vertices, it follows that $G_1$ is either a path of order three, herein denoted by $P_3$, or a matching with two edges, herein denoted by $2K_2$. A related folklore result is that $\hr(K_{1,m},K_{1,n})=m+n-1$, where $K_{1,n}$ is a star with $n$ edges. Faudree and Sheehan proved the following result, where $K_n$ is a complete graph on $n$ vertices.
\begin{theorem}[Faudree and Sheehan~\cite{Faudree1983Sizea}]\label{faudreetheorem}
	For a positive integer $n$ with $n\ge 2$,  $$\hr(P_3,K_n)=2(n-1)^2.$$
\end{theorem}
Faudree and Sheehan also determined $\hr(K_{1,k},K_\ell+\overline{K_n})$ for $\ell \ge k \ge 2$ and every sufficiently large $n$. Here, $K_\ell+\overline{K_n}$ is formed from a complete graph $K_\ell$ and $n$ isolated vertices by joining every isolated vertex to every vertex of $K_\ell$. More precisely, they have the following result.
\begin{theorem}[Faudree and Sheehan~\cite{Faudree1983Sizea}]
	For positive integers $k$ and $\ell$ with $k\ge \ell\ge 2$ and sufficiently large $n$, $$\hr(K_{1,k},K_\ell+\overline{K_n})=\binom{k(\ell-1)+1}{2}+(k(\ell-1)+1)(n+k-1).$$
\end{theorem}

Let $nK_2$ be a matching with $n$ edges. Erd\H{o}s and Faudree~\cite{Erdoes1984Size} obtained quite a few size Ramsey numbers involving matchings. They established that for all positive integers $m$ and $n$, $\hat{r}(nK_2,K_{1,m})=mn$, $\hat{r}(nK_2,P_4)=\cei{5n/2}$, and $\hat{r}(nK_2,G)=n\vert E(G)\vert $ for any connected graph $G$ with $\vert V(G)\vert \le 4$ except for $P_4$. Moreover, they showed $\hat{r}(nK_2,P_5)=3n$ if $n$ is even, and $\hat{r}(nK_2,P_5)=3n+1$ if $n$ is odd. They also explored the size Ramsey number of a matching versus a complete graph and proved that for $m\ge 4n-1$, we have $\hat{r}(nK_2,K_m)=\binom{m+2n-2}{2}$. If the matching has two edges, they showed the following theorem.

\begin{theorem}[Erd\H{o}s and Faudree~\cite{Erdoes1984Size}]\label{erdostheorem}
	For all $m, n \ge 2$, we have
	$$\hr(2K_2,K_m)= \begin{cases}\binom{m+2}{2}, & \text{for}\ m\ge 6, \\ 2\binom{m}{2}, & \text{for}\ 2\le m\le 5,\end{cases} \qquad\text{and}\qquad\hr(2K_2,K_{m,n})=mn+m+n.$$
\end{theorem}

If both $G_1$ and $G_2$ are small-order graphs, Faudree and Sheehan~\cite{Faudree1983Size} listed the size Ramsey numbers for all pairs of graphs with at most four vertices; and Lortz and Mengersen~\cite{Lortz2021Size} determined the size Ramsey number $\hr(P_3,H)$ for all graphs $H$ of order five.

The fan $F_n$ is a graph of order $2n+1$ with $n$ triangles sharing a common vertex. It is one of the most studied graphs in graph Ramsey theory. For Ramsey numbers involving fans, we refer the reader to Radziszowski's dynamic survey~\cite{Radziszowski2021Small}. Baskoro~et al.~\cite{Baskoro2006Upper} investigated the size Ramsey number of $P_3$ versus a fan and proved the upper bound $\hr(P_3,F_n)\le 6n+2$. They conjectured that this upper bound is also a lower bound (and hence equality holds). Our first result gives the exact value of $\hat{r}(P_3, F_n)$ for every positive integer $n$ and, as a by-product, disproves Baskoro~et al.'s conjecture.

\begin{theorem}\label{pathfan} For every positive integer $n$, we have
	$$\hr(P_3,F_n)= \begin{cases}10, & \text{if}\ n=2; \\4n+4, & \text{if}\ n\ \text{is odd}; \\ 4n+5, & \text{if}\ n\ \text{is even and}\ n\not=2. \end{cases}$$
\end{theorem}

Turning to a matching with two edges versus a fan, we know that $\hr(2K_2,F_1)=\hr(2K_2,K_3)=6$ by Theorem \ref{erdostheorem}. For $n\ge 2$, we determine  $\hr(2K_2,F_n)$ completely.

\begin{theorem}\label{matchingfan}For every positive integer $n$, we have
	$$\hr(2K_2,F_n)= \begin{cases}6n, & \text{if}\ 1\le n\le 2; \\ 5n+3, & \text{if}\ n\ge 3. \end{cases}$$
\end{theorem}

Even though the problem of determining the size Ramsey number $\hr(P_3,P_n)$ is far from being settled, one can easily show that $\hr(2K_2,P_n)=n+1$ for $n\ge 3$. Moreover, Vito and Silaban~\cite{Vito2022Two} calculated $\hr(2K_2,nP_m)$ for $2\le n\le 4$ and $m\ge 3$, where $nP_m$ is the disjoint union of $n$ copies of $P_m$. They conjectured that $\hr(2K_2,nP_m)=\min\{nm+1, (n+1)(m-1)\}$ for $n\ge 5$ and $m\ge 3$. Our third result confirms this conjecture.

\begin{theorem}\label{matchingpath} For all positive integers $n$ and $m$ with $m\ge 2$,
	$$\hr(2K_2,nP_m)=\min\{nm+1, (n+1)(m-1)\}.$$
\end{theorem}

We will prove Theorems \ref{pathfan}, \ref{matchingfan}, and \ref{matchingpath} in Sections \ref{section2}, \ref{section3}, and \ref{section4}, respectively. We recall some more notation at the end of this section. The graph $G[N(v)]$ is the subgraph of $G$ induced by the neighbors of $v$. The graph $G-v$ is the subgraph obtained from $G$ by deleting the vertex $v$ and all edges incident to $v$. We use $nG$ to denote the disjoint union of $n$ copies of $G$. The graph $K_1+nG$ is obtained by adding one new vertex to $nG$ and joining it to every vertex of $nG$. This additional vertex is called the \emph{center} of $K_1+nG$. A cycle of order $n$ is denoted by $C_n$. A wheel $W_n$ is the graph $K_1+C_n$, and a fan $F_n$ is the graph $K_1+nK_2$. The $2n$ edges of $F_n$ that are incident to $v$ are called \emph{spokes}, and their set is denoted by $E_1(F_n)$; the $n$ edges that are not incident to $v$ are called \emph{rim edges}, and their set is denoted by $E_2(F_n)$.

\section{Proof of Theorem \ref{pathfan}}\label{section2}

\begin{figure}[htbp]
	\centering
	\begin{minipage}{0.4\linewidth}
		\centering
		\includegraphics[width=1\textwidth]{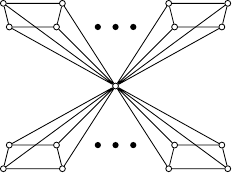}
		\caption{The graph $K_1+nC_4$}
	\end{minipage}
	\hspace{50pt}
	\begin{minipage}{0.4\linewidth}
		\centering
		\includegraphics[width=1\textwidth]{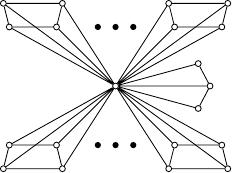}
		\caption{The graph $K_1+(nC_4\cup P_3)$}
	\end{minipage}
\end{figure}

Notice that $\hr(P_3,F_1)=\hr(P_3,K_3)=8$ by Theorem \ref{faudreetheorem}, and $\hr(P_3,F_2)=10$ by Lortz and Mengersen~\cite{Lortz2021Size}. So we assume that $n\ge 3$. To simplify the notation, let us restate the upper bounds as $\hr(P_3,F_{2n-1})\le 8n$ and $\hr(P_3,F_{2n})\le 8n+5$. We first consider the graph formed by $n$ copies of $W_4$ sharing a common center, denoted by $K_1+nC_4$. We show that $K_1+nC_4\to (P_3, F_{2n-1})$. For any red-blue edge-coloring of $K_1+nC_4$, suppose there is no red $P_3$, we need to find a blue $F_{2n-1}$. Let $v$ be the center of $K_1+nC_4$. Then $v$ is incident to at most one red edge, since otherwise there is a red $P_3$. It follows from $C_4\to (P_3,2K_2)$ that $nC_4\to (P_3,(2n)K_2)$. In other words, the graph $nC_4$ contains a blue matching with $2n$ edges, which together with $v$ forms an $F_{2n}$. This fan has at most one red edge, which is incident to $v$. Thus, there is a blue $F_{2n-1}$. It follows that $\hr(P_3,F_{2n-1})\le 8n$. To prove $\hr(P_3,F_{2n})\le 8n+5$, we add a new $P_3$, say $v_1v_2v_3$, to the graph $K_1+nC_4$, and join $v$ to $v_1,v_2,v_3$, respectively. From the facts $C_4\to (P_3,2K_2)$ and $P_3\to (P_3,K_2)$, we see that $nC_4\cup P_3\to (P_3,(2n+1)K_2)$. Hence, there is a blue matching $(2n+1)K_2$ in $G[N(v)]$, which together with $v$ forms an $F_{2n+1}$. It contains at most one red edge, which is incident to $v$. So there is a blue $F_{2n}$. Thus, $\hr(P_3,F_{2n})\le 8n+5$.

Now we prove $\hr(P_3,F_n)\ge 4n+4+\epsilon$ for $n\ge 3$. Here, $\epsilon$ is a Kronecker delta function, which means that $\epsilon=0$ if $n$ is odd, and $\epsilon=1$ if $n$ is even. For any graph $G$ with $4n+3+\epsilon$ edges, we will find an edge-coloring of $G$ such that $G$ contains neither a red $P_3$ nor a blue $F_n$. When coloring the edges of $G$ in the sequel, we guarantee that no red $P_3$ will appear in $G$. Suppose to the contrary that $G$ contains a blue $F_n$, we can always find a contradiction. We first have the following claim on the vertex with maximum degree.

\begin{claim}\label{centerclaim}
	The graph $G$ contains exactly one vertex whose degree is at least $2n$.
\end{claim}

\begin{proof}
	Suppose that there are two vertices $u$ and $v$ such that $d(u)\ge 2n$ and $d(v)\ge 2n$. We will find a matching and color its edges red in the sequel. Hence, we always assume that the graph induced by all uncolored edges contains an $F_n$. This is because, if such an $F_n$ does not exist at some step, we color all uncolored edges blue and a contradiction follows.
	
	If there is an edge joining $u$ and $v$, we color it red. Let $G_1$ be an $F_n$ all of whose edges are uncolored. If the center of $G_1$ is not $u$ or $v$, then at least $2n-2$ edges incident to $u$ are not in $G_1$, and so are the edges incident to $v$. Since $uv$ is a possible edge, the size of $G$ is at least $\vert E(G_1)\vert +(2n-2)+(2n-2)-1$, which is greater than $4n+3+\epsilon$ for $n\ge 3$, a contradiction. Thus, every copy of $F_n$ has center either $u$ or $v$. Without loss of generality, let $u$ be the center of $G_1$. We color all its rim edges red. Now let $G_2$ be an $F_n$ all of whose edges are uncolored. Recall that we have already colored the edge $uv$ red if it exists. If the center of $G_2$ is $u$, then any edge incident to $v$ is not contained in $G_2$. The size of $G$ is at least $\vert E(G_2)\vert +n+2n$, which contradicts the fact that $\vert E(G)\vert =4n+3+\epsilon$. If the center of $G_2$ is $v$, then any edge of $G_1$ is not contained in $G_2$. The size of $G$ is at least $\vert E(G_1)\vert +\vert E(G_2)\vert $, again a contradiction.
\end{proof}

Let $v$ be the vertex with $d(v)=\Delta(G)$. By Claim \ref{centerclaim}, the center of $F_n$ has to be $v$ for any copy of $F_n$ in $G$. Among all vertices with maximum degrees in the graph $G[N(v)]$, we choose such a vertex $u$ that $G[N(v)\setminus u]$ contains a component that is neither a $C_4$ nor a $P_5$. If such a vertex does not exist, then we choose an arbitrary vertex as $u$ which has maximum degree in $G[N(v)]$. The vertex $u$ has the following properties.

\begin{claim}\label{du1}
	The vertex $u$ has at least two neighbors in $G[N(v)]$.
\end{claim}

\begin{proof}
	Suppose not, then $\Delta(G[N(v)])\le 1$. For each copy of $F_n$ in $G$, we color all its rim edges red. All rim edges from different copies of $F_n$ form a red matching, since otherwise $\Delta(G[N(v)])\ge 2$. After coloring the remaining edges blue, we see that $G$ contains neither a red $P_3$ nor a blue $F_n$ as a subgraph. This contradicts our assumption.
\end{proof}

\begin{claim}\label{du2}
	If $n$ is odd, then $d(u)=3$. If $n$ is even, then $3\le d(u)\le 4$. Moreover, the graph $G[N(v)\setminus u]$ has maximum degree at most two.
\end{claim}

\begin{proof}
	We color the edge $uv$ red. Thus each copy of a blue $F_n$ does not contain any edge incident to $u$. If $G-u$ contains no $F_n$, we color all other edges blue and our proof is done. If $G-u$ contains a copy of $F_n$, denoted by $H_1$, we color all edges of $E_2(H_1)$ red, where $E_2(H_1)$ is the set of rim edges of $H_1$. If the graph induced by all uncolored edges contains no $F_n$, then we color all uncolored edges blue and our proof is done. If the graph induced by all uncolored edges contains an $F_n$, denoted by $H_2$, then $G$ has at least $d(u)+\vert E_2(H_1)\vert +\vert E(H_2)\vert $ edges, which is $4n+d(u)$. Combining the fact $\vert E(G)\vert =4n+3+\epsilon$ and Claim \ref{du1}, we have $d(u)=3$ if $n$ is odd, and $3\le d(u)\le 4$ if $n$ is even. If $n$ is even and $d(u)=4$, then the $4n$ edges of $E(G-u)$ consists of  all edges from $E_2(H_1)$ and $E(H_2)$. Thus, no matter $d(u)=4$ or $d(u)=3$, the graph $G[N(v)\setminus u]$ has maximum degree at most two.
\end{proof}

Recall that we have colored $uv$ red. Now we color all other edges incident to $u$ or $v$ blue. By Claim \ref{du2}, the graph $G[N(v)\setminus u]$ consists of a disjoint join of cycles, paths, and isolated vertices. For a cycle of length $\ell$ with $\ell\ge 3$, and for each $i$ with $0\le i\le \flo{\ell/3}-1$, starting from any edge, we color the $(3i+1)$-th edge along the cycle with red color, and the other edges blue. Thus, a maximum blue matching of $C_{\ell}$ has $\cei{\ell/3}$ edges under this coloring. For a path of length $\ell$, and for each $i$ with $0\le i\le \flo{\ell/3}-1$, starting from any pendant edge, we color the $(3i+1)$-th edge along the path with red color, and the other edges blue. Thus, if $\ell\ge 2$, a maximum blue matching of $P_{\ell+1}$ has $\cei{\ell/3}$ edges under this coloring. If $\ell\le 1$, $P_{\ell+1}$ has no blue edge.

We color the remaining uncolored edges blue, and prove that $G$ contains no blue $F_n$. If not, $G[N(v)\setminus u]$ has a blue matching with $n$ edges. Assume that the cycles in $G[N(v)\setminus u]$ have length $c_1,c_2,\ldots, c_s$, respectively, and the paths of length at least two in $G[N(v)\setminus u]$ have length $p_1,p_2,\ldots, p_t$, respectively. Then the size of a maximum blue matching in $G[N(v)\setminus u]$ is $$\sum_{i=1}^{s}\cei{c_i/3}+\sum_{j=1}^{t}\cei{p_j/3},$$ which is at least $n$. And the size of $G$ is at least $$2\sum_{i=1}^{s}c_i+\sum_{j=1}^{t}(2p_j+1)+d(u),$$ which is at most $4n+3+\epsilon$. It follows that $\sum_{i=1}^{s}4\cei{c_i/3}+\sum_{j=1}^{t}4\cei{p_j/3}+3+\epsilon\ge 4n+3+\epsilon \ge 2\sum_{i=1}^{s}c_i+\sum_{j=1}^{t}(2p_j+1)+d(u)$. Hence we have
\begin{equation}
	\sum_{i=1}^{s}(4\cei{c_i/3}-2c_i)+\sum_{j=1}^{t}(4\cei{p_j/3}-2p_j-1)+3+\epsilon\ge d(u). \label{equ}
\end{equation} Note that $c_i\ge 3$ for $1\le i\le s$, and $p_j\ge 2$ for $1\le j\le t$. It is easy to check that if $c_i=4$, then $4\cei{c_i/3}-2c_i=0$; if $c_i=3$ or $c_i\ge 5$, then $4\cei{c_i/3}-2c_i\le -2$; if $p_j=2$ or $p_j=4$, then $4\cei{p_j/3}-2p_j-1=-1$; if $p_j=3$ or $p_j\ge 5$, then $4\cei{p_j/3}-2p_j-1\le -3$. By Claim \ref{du2}, $G[N(v)\setminus u]$ consists of a disjoint union of $C_4$ if $n$ is odd. If $n$ is even, then $G[N(v)\setminus u]$ consists of a disjoint union of $C_4$, and at most one component which is a path with length two or four.

By Claims \ref{du1} and \ref{du2}, $u$ has two or three neighbors in $G[N(v)]$. We first consider the case that $G[N(v)\setminus u]$ consists of a disjoint union of $C_4$. If $u$ has two neighbors in $G[N(v)]$, denoted by $w_1,w_2$, then $w_1$ has three neighbors in $G[N(v)]$, contradicting the choice of $u$. If $u$ has three neighbors in $G[N(v)]$, denoted by $w_1,w_2,w_3$, then each $w_i$ has three neighbors in $G[N(v)]$ for $1\le i\le 3$. Moreover, there is a $w_i$ for some $1\le i\le 3$ such that $G[N(v)\setminus w_i]$ contains a component which is neither a $C_4$ nor a $P_5$, which contradicts the choice of $u$ again. Next we consider the case that $G[N(v)\setminus u]$ contains a path of length two as a component. In this case, $n$ is even. To find a blue matching with $n$ edges in $G[N(v)]$, it must contain at least $n/2$ disjoint copies of $C_4$. Thus, $G$ has at least $4n+5$ edges, a contradiction. Finally, we consider the case that $G[N(v)\setminus u]$ contains a path of length four as a component. By the inequality \eqref{equ}, $n$ is even, and $u$ has exactly two neighbors in $G[N(v)]$, denoted by $w_1, w_2$. If $w_1$ has two neighbors in $G[N(v)\setminus u]$, then $w_1$ has three neighbors in $G[N(v)]$, which contradicts the choice of $u$. Thus, $w_1$ has to be an end vertex of the path $P_5$, and so does $w_2$. Thus, all components of $G[N(v)]$ are cycles, one of which is $C_6$, and the others are $C_4$. By the choice of $u$, we would choose a vertex of $C_4$ rather than a vertex of $C_6$ to be $u$, which contradicts the fact that $u$ belongs to $C_6$. This completes our proof.

\section{Proof of Theorem \ref{matchingfan}}\label{section3}

\begin{figure}[htbp]
	\centering
	\includegraphics[width=0.8\textwidth]{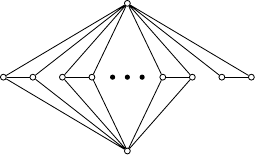}
	\caption{The graph $H$ satisfying $H\to (2K_2,F_n)$}
	\label{figure3}
\end{figure}

We first prove that $\hr(2K_2,F_n)=5n+3$ for $n\ge 3$. To prove the upper bound, we first consider a matching with $n$ edges, and join two new vertices $u$ and $v$ to every vertex of the matching respectively. These are two fans $F_n$ sharing the same rim edges. Then we add a new edge $u_1u_2$, and join $u$ to $u_1$ and $u_2$. Denote this graph by $H$ (see Figure \ref{figure3}). We prove that $H\to (2K_2,F_n)$. For any red-blue edge-coloring of $H$, if there is no red $2K_2$, then the red edges form either a star or a triangle. If the red edges form a triangle, or a star whose center is not $u$, then there is always a blue $F_n$ with $u$ as its center. If the red edges form a star whose center is $u$, then there is a blue $F_n$ with $v$ as its center. So $H\to (2K_2,F_n)$, and $\hr(2K_2,F_n)\le 5n+3$ for $n\ge 3$.

Now we turn to the lower bound. Suppose to the contrary that there is a graph $G$ with $5n+2$ edges and $G\to (2K_2,F_n)$ for $n\ge 3$. For any vertex $x$, we claim that the graph $G-x$ must contain an $F_n$. If not, we color all edges incident to $x$ red and the other edges of $G$ blue. This coloring implies that $G\not\to (2K_2,F_n)$, which contradicts our assumption. Let $v$ be a vertex of $G$ with maximum degree. Then $d(v)\ge 2n$. Notice that $G-v$ contains an $F_n$, denoted by $H_1$. We denote the center of $H_1$ by $u$.

If $d(v)\ge 2n+3$,  we have $\vert E(G)\vert \ge d(v)+\vert E(H_1)\vert \ge 5n+3$, a contradiction. If $2n+1\le d(v)\le 2n+2$, then $d(v)+\vert E(H_1)\vert \ge 5n+1$. Since $\vert E(G)\vert =5n+2$, except for at most one edge, every edge of $G$ is either incident to $v$, or an edge of $H_1$. As $n\ge 3$, $v$ is the only vertex with degree at least $2n$ in $G-u$. Since $G-u$ contains an $F_n$, the center of this $F_n$ must be $v$. Thus a matching with $n$ edges is contained in $G[N(v)]$. Consequently, all but at most one rim edge of $H_1$ is contained in $G[N(v)]$. There is a triangle that contains both $v$ and a rim edge of $H_1$. We color its edges red, and all other edges of $G$ blue. It is not difficult to check that there is no blue $F_n$ with $u$ as its center. If there is a blue $F_n$ with $v$ as its center, then $d(v)=2n+2$. But there are only $n-1$ blue edges in $G[N(v)]$ in this case. As a result, $G$ contains no blue $F_n$, a contradiction. If $d(v)=2n$, then $d(u)=2n$. Since $\vert E(G)\vert =5n+2$, except for at most two edges, every edge of $G$ is either incident to $v$, or an edge of $H_1$. Hence for every vertex $x\in V(G)\setminus \{u,v\}$, we have $d(x)\le 5<2n$. That is to say, each copy of $F_n$ must have center either in $u$ or in $v$. If $N(u)\cap N(v)=\emptyset$, then $\vert E(G)\vert \ge 3n+3n>5n+2$, because $G-u$ contains an $F_n$. If $N(u)\cap N(v)\not=\emptyset$, then we choose a common neighbor of $u$ and $v$, denoted by $w$. We color all edges incident to $w$ red, and the other edges blue. Thus, there is no red $2K_2$ or blue $F_n$, a final contradiction.

Next we show $\hr(2K_2,F_2)=12$. The upper bound of $\hr(2K_2,F_2)$ follows from the fact that $2F_2\to (2K_2,F_2)$. For the lower bound, we use a similar argument as in the case $n\ge 3$. Suppose by contradiction that there is a graph $G$ of size $11$ satisfying $G\to (2K_2,F_2)$. For any vertex $x$, the graph $G-x$ contains an $F_2$, since otherwise we color all edges incident to $x$ red and the other edges blue, which implies that $G\not \to (2K_2,F_2)$. We use the same notation as above for simplicity. Let $v$ be a vertex with maximum degree, $H_1$ a fan $F_2$ in $G-v$, and $u$ the center of $H_1$. If $d(v)\ge 6$, then $\vert E(G)\vert \ge d(v)+\vert E(H_1)\vert =12$, a contradiction. If $d(v)=5$, since $\vert E(G)\vert =d(v)+\vert E(H_1)\vert =11$, every edge of $G$ is either incident to $v$, or an edge of $H_1$. It is easy to find an edge-coloring of $G$ such that there is neither a red $2K_2$ nor a blue $F_2$. If $d(v)=4$, then $d(u)=4$. Except for at most one edge, denoted by $e_1$, every edge of $G$ is either incident to $v$, or an edge of $H_1$. Any $F_2$ contained in $G-u$ must have $v$ as its center, since $v$ is the only vertex with degree at least four in $G-u$. Let $e_2$ and $e_3$ be the rim edges of $H_1$. If $e_1$ is incident with $e_2$ or $e_3$, say, $e_1$ and $e_2$ share a common end vertex $w$, then we color all edges incident to $w$ red, and the other edges blue. Hence $G\not \to (2K_2,F_2)$. If $e_1,e_2,e_3$ form a matching, then at least one of $e_2,e_3$ is contained in $G[N(v)]$. Assume without loss of generality that $e_2$ is contained in $G[N(v)]$, then we color $e_2$ red and the other edges blue. Again, we have $G\not \to (2K_2,F_2)$.

\section{Proof of Theorem \ref{matchingpath}}\label{section4}

The theorem for $m=2$ is easily verified, because $\hr(n_1K_2,n_2K_2)=n_1+n_2-1$. So we assume that $m\ge 3$. We need the following two lemmas for our argument. The first one gives us the upper bound, and the second gives us the lower bound for any connected graph $G$.
\begin{lemma}[Vito and Silaban~\cite{Vito2022Two}]\label{upperbound}
	For $n\ge 1$ and $m\ge 3$, we have $\hr(2K_2,nP_m)\le \min\{nm+1, (n+1)(m-1)\}$.
\end{lemma}

The proof of Lemma \ref{upperbound} is straightforward. The upper bound $nm+1$ follows from the fact that $C_{nm+1}\to (2K_2,nP_m)$, and the upper bound $(n+1)(m-1)$ follows from the fact that $(n+1)P_m\to (2K_2,nP_m)$.

\begin{lemma}[Vito and Silaban~\cite{Vito2022Two}]\label{connected}
	Let $m$ and $n$ be two positive integers and $m\ge 3$. For any connected graph $G$ with $nm$ edges, we have $G\not\to (2K_2,nP_m)$.
\end{lemma}

Let $N=\min\{nm+1, (n+1)(m-1)\}$. We prove the lower bound by contradiction. Suppose to the contrary that $G$ is a graph with $N-1$ edges and $G\to (2K_2,nP_m)$. By Lemma \ref{connected} and $\vert E(G)\vert \le nm$, $G$ is a disconnected graph. Let $G_1, G_2, \ldots, G_t$ be all the components of $G$. For each $i$ with $1\le i\le t$, let $a_i$ be the largest integer such that $G_i$ contains the disjoint union of $a_i$ copies of $P_m$. If $\sum_{i=1}^t a_i<n$, we color all edges blue, and then $G\not\to (2K_2,nP_m)$. If $\sum_{i=1}^t a_i>n$, then $\vert E(G)\vert \ge (m-1)\sum_{i=1}^t a_i\ge (n+1)(m-1)$, which contradicts the size of $G$. Thus we have $\sum_{i=1}^t a_i=n$. If there exists some $i$ with $1\le i\le t$ such that $G_i\not\to (2K_2, a_iP_m)$, then $G_i$ has a coloring that $G_i$ contains neither a red $2K_2$ nor a blue $a_iP_m$. We extend this coloring to a coloring of $G$ by coloring all other edges blue. The graph $G$ contains neither a red $2K_2$ nor a blue $nP_m$ under this coloring. This contradiction implies that $G_i\to (2K_2, a_iP_m)$ for each $i$ with $1\le i\le t$. Applying Lemma \ref{connected} to each component $G_i$, we have $\vert E(G)\vert =\sum_{i=1}^{t}\vert E(G_i)\vert \ge \sum_{i=1}^{t}(a_im+1)=m\sum_{i=1}^{t}a_i+t=mn+t\ge nm+2>\vert E(G)\vert $, a final contradiction.

\bmhead{Acknowledgments}

The second author was partially supported by NSFC under grant numbers 11601527, while
the third author was partially supported by NSFC under grant numbers 12161141003 and 12271246.




\end{document}